\long\global\def\C#1\F{{}}
\documentclass[12pt, leqno]{article}
\usepackage{amsmath, amsthm, amssymb, graphicx, verbatim}
\usepackage[round]{natbib}
\swapnumbers

\newtheorem{theorem}{Theorem}
\newtheorem{lemma}[theorem]{Lemma}
\newtheorem{definition}[theorem]{Definition}
\newtheorem{remark}[theorem]{Remark}
\newtheorem{remarks}[theorem]{Remarks}
\newtheorem{example}[theorem]{Example}

\newtheorem{gauss graph}[theorem]{The Graph of Gauss's Puzzle}
\newtheorem{arrang hyper}[theorem]{An Arrangement of Hyperplanes}

\newtheorem{counterexample}[theorem]{A Counterexample}
\newtheorem{note}[theorem]{Note}
\newtheorem{notes}[theorem]{Notes}

\renewenvironment{proof}{{\sc Proof.}}{\EOP\wl}

\def\vtl{\vskip 1mm}
\def\tl{\vskip 2mm}
\def\wl{\vskip 4mm}
\def\es{\varnothing}

\def\SB{\subseteq}

\def\ss{\sigma}

\def\Ree{\mathbb R}

\def\Span{\mathbb S}

\def\EQ{\Longleftrightarrow}
\def\eq{\Leftrightarrow}
\def\EOP{\phantom{a}\hfill $\square$}

\def\XXX{{\cal X}}
\def\AAA{{\cal A}}
\def\FFF{{\cal F}}

\def\GGG{{\cal G}}
\def\HHH{{\cal H}}

\def\KKK{{\cal K}}
\def\BBB{{\cal B}}
\def\VVV{{\cal V}}
\def\WWW{{\cal W}}

\def\POW{\mathfrak P}

\def\begeq{\begin{equation}}
\def\edeq{\end{equation}}
\def\st{{\;\vrule height8pt width0.7pt depth2.5pt\;}}

\def\sqr#1#2{{\vcenter{\vbox{\hrule height.#2pt
          \hbox{\vrule width.#2pt height#1pt \kern#1pt
              \vrule width.#2pt}
           \hrule height.#2pt}}}}
\def\square{\mathchoice\sqr56\sqr56\sqr{2.1}3\sqr{1.5}3}
\def\roster{\begin{enumerate}}
\def\endroster{\end{enumerate}}

\def\fp{\noindent}

\reversemarginpar

\let\slantedexample\example
\def\example{\slantedexample\rm}
\let\slantedremark\remark
\def\remark{\slantedremark\rm}
\let\slantedremarks\remarks
\def\remarks{\slantedremarks\rm}
\let\slanteddefinition\definition
\def\definition{\slanteddefinition\rm}
\let\slantednote\note
\def\note{\slantednote\rm}
\let\slantednotes\notes
\def\notes{\slantednotes\rm}

\makeatletter
\long\def\@makecaption#1#2{%
  \vskip\abovecaptionskip
  \sbox\@tempboxa{\small #1: \sc #2}%
  \ifdim \wd\@tempboxa >\hsize
    \small #1: \sc #2\par
  \else
    \global \@minipagefalse
    \hb@xt@\hsize{\hfil\box\@tempboxa\hfil}%
  \fi
  \vskip\belowcaptionskip}
\makeatother

\begin{document}
\pagestyle{plain}

\title{\vspace*{-1.5cm}\Large On Verifying and Engineering the Well-gradedness\\ of a Union-closed Family\thanks{We are grateful to Eric Cosyn and Chris Doble for some useful discussions. We also thank the two referees for their detailed comments and suggestions, many of which were incorporated in the final version. }}
\author{David Eppstein\thanks{Computer Science Department,
University of California, Irvine, CA 92697.
}
\hskip 1cm Jean-Claude Falmagne\thanks{Dept.~of Cognitive Sciences,
University of California, Irvine, CA 92697. Phone: (949) 433 2735.}\\
\normalsize \{eppstein, jcf\}@uci.edu\\
\\
Hasan Uzun\thanks{ALEKS Corporation.}\\
\normalsize huzun@aleks.com}

\date{\small \today}
\maketitle

\thispagestyle{empty}

\begin{abstract}
\fp
Current techniques for generating a knowledge space, such as QUERY, guarantees  that the resulting structure is closed under union, but not that it satisfies wellgradedness, which is one of the defining conditions for a learning space. We give necessary and sufficient conditions on the base of a union-closed set family that ensures that the family is well-graded. We consider two cases, depending on whether or not the family contains the empty set. We also provide algorithms for efficiently testing these conditions, and for augmenting a set family in a minimal way to one that satisfies these conditions.
\end{abstract}

\section*{Introduction}
A family of sets  $\FFF$ is well-graded if any two sets in $\FFF$ can be connected by a sequence of sets formed by single-element insertions and deletions, without redundant operations, such that all intermediate sets in the sequence belong to $\FFF$. The family $\FFF$ is called $\cup$-closed if it is closed under union. (Formal definitions are given in our next section.) Well-graded families are of interest for theorists in several different areas of combinatorics, as various families of sets or relations are well-graded. For example, Theorems 2 and 4 in \citet{bogar73} imply that the family of all partial orders on a finite set is well-graded. The same property of well-gradedness is shared by other families, such as the semiorders, the interval orders, and the biorders, again on finite sets \citep{doign97}.  Via representation theorems, this concept also applies to the \emph{partial cubes}, to wit, graphs isometrically embeddable into hypercubes  \citep[][]{graha71,djoko73,winkl84,imric00}, and to the \emph{oriented media} which are semigroups of transformations satisfying certain axioms \citep{falma97a,eppst07b}.

When the family $\FFF$ is well-graded, $\cup$-closed and contains the empty set, one obtains an object variously called an \emph{antimatroid} \citep{KorLovSch-91}, a  \emph{learning space} \citep{cosyn08,falma06}, or a \emph{well-graded knowledge space}
 \citep{doignon:85}. The monograph of \citet{doign99} contains a comprehensive account of this topic.
Learning spaces are applied in mathematical modeling of education. In such cases, the ground set is the collection of problems, for example in elementary arithmetic, that a student must learn to solve in order to master the subject. The family $\FFF$ contains then all the subsets forming the feasible \emph{knowledge states}. In practice, the size of such a family is quite large, typically containing millions of states\footnote{For a ground set  that may contain a couple of hundreds of problem types.}, which raises the problem of summarizing $\FFF$ efficiently. An obvious choice for this purpose is the \emph{base} of that family, namely the unique minimal subset of $\FFF$ whose completion via all possible unions gives back 
$\FFF$.

For various reasons, when building a learning space in practice, one may fall short of some sets to achieve well-gradedness, a property regarded as essential for promoting efficient learning  \citep[see the axiomatization of][]{cosyn08}. This raises the problems of uncovering possibly missing sets, and completing the family economically and/or optimally. 
These  considerations inspired the work presented here.

We solve the following problems for a finite family $\GGG$ of finite sets.
\begin{enumerate}
\item  Find  necessary and sufficient conditions for $\GGG$ to be the base of a well-graded $\cup$-closed family of sets.
\item Find such conditions  when the well-graded $\cup$-closed family of sets is known to be a learning space, that is, the family contains the empty set. (These conditions may be simpler than in Case 1.)
\item Provide efficient algorithms for testing these conditions on a family 
$\GGG$ and uncovering possibly missing sets.
 (Different algorithms may be used in Problems 1 and 2.)
\item Supposing that some family $\GGG$ fails to satisfy the conditions in Problems 1 or~2, provide  algorithms for modifying $\GGG$ in some optimal sense to yield a family $\GGG^*$ satisfying such conditions.
\end{enumerate}

Except for the passing remark involving Counterexample \ref{counter wg no atom}, only finite sets are considered in this paper.

\section*{Background and Preparatory Results}

\begin{definition}
\label{def k-path}
Let $\FFF$ be a family of subsets of a set $\XXX$. A {\sl tight path between} two distinct sets $P$ and $Q$ (or from $P$ to $Q$) in $\FFF$ is a sequence
$P_0=P, P_1,\ldots,P_n = Q$ in $\FFF$
such that  $d(P,Q)= |P\bigtriangleup Q|= n$ and  $d(P_i,P_{i+1}) = 1$ for $0\leq i\leq n-1$. 

The family $\FFF$ is  {\sl well-graded} or a {\sl wg-family} if there is a tight path between any two of its distinct sets.
\citep[cf.][]{doign97, falma97b}\footnote{This concept was introduced earlier under a different name; see \citet{kuzmi75}, \citet{ovchi80}.}.
\end{definition}

\begin{definition}
\label{U closure}
A family of sets $\FFF$ is {\sl closed under union}, or  {\sl $\cup$-closed} if for any nonempty\footnote{For some authors, the subfamily $\GGG$ may be empty, with $\cup\es = \es$. So, a $\cup$-closed family automatically contains the empty set. We do not use this convention here.} $\GGG\SB\FFF$ we have
$\cup \GGG\in\FFF$. A well-graded family closed under union and containing the empty set is a \emph{learning space}.
\end{definition} 
 \begin{definition}\label{span-base}  The {\sl span} of a family of sets $\GGG$ is the family $\GGG^\dagger$ containing any set which is the union  of some subfamily\footnote{Contrary to the convention used by \citet{doign99}, the empty subfamily of $\GGG$ is not allowed; so $\es\in \Span(\GGG)$ only if $\es\in\GGG$.} of $\GGG$. In such a case, we write $\Span (\GGG) = \GGG^\dagger$  and we say that $\GGG$ {\sl spans} $\GGG^\dagger$. By definition $\Span(\GGG)$ is thus $\cup$-closed. A {\sl base} of a $\cup$-closed family 
$\FFF$  is a minimal subfamily $\BBB$ of $\FFF$ spanning $\FFF$ (where `minimal' is meant with respect to set inclusion: if $\Span (\HHH) = \FFF$ for some $\HHH \SB \BBB$, then $\HHH = \BBB$). Notice that if $\es\in\FFF$, we must have $\es\in \BBB$, with $\cup\{\es\} = \es$. In such a case, we use the abbreviation $\check \BBB = \BBB\setminus\{\es\}$. Note that a family $\GGG$ spanning a family $\FFF$ is a base of $\FFF$ if and only if none of the sets in $\GGG$ is the union of some other sets in $\GGG$.
\end{definition}

Any finite $\cup$-closed family has a base, which is unique.
 This uniqueness property of the base also holds in the infinite case but some infinite families have no base: take, for example, the collection of all open sets of $\Ree$ or Counterexample~\ref{counter wg no atom}.  \citep[In this regard, see][Theorems 1.20 and 1.22.]{doign99} 
 
The following lemma is a key tool, as it allows us to infer the wellgradedness of a family from that of its base.
 
\begin{lemma}
\label{span is wg}
The span of a wg-family is well-graded.
\end{lemma}

\begin{proof}
Let $\Span (\GGG)$ be the span of some wg-family $\GGG$. Take any two distinct $X,Y$ in $\Span(\GGG)$. 
Since $\Span (\GGG)$ is $\cup$-closed by definition,  $X\cup Y$ is in $\Span (\GGG)$ and we have  
$
d(X,Y) = d(X,X\cup Y) + d(X\cup Y, Y).
$
Accordingly, it suffices to prove that there is in $\Span(\GGG)$ a tight path
\begin{equation}
X_1 = X,X_2,\ldots,X_n = X\cup Y,
\end{equation}
with in fact $X_i\subset X_{i+1}$, $1\leq i\leq n-1$.
By definition of the span, there exists finite $\HHH,\KKK \SB \GGG$ such that $X=\cup\HHH$ and $Y=\cup\KKK$. Without loss of generality (exchanging the roles of $X$ and $Y$ if needed), we can assume that there exists some $K\in\KKK$ such that $K\setminus X\neq\es$. Choose $H\in \HHH$ arbitrarily. By the wellgradedness of $\GGG$, there is a tight path $H_1 = H,\ldots,H_m = K$. Let $k$ be the first index such that $H_k\setminus X\neq \es$.
(Such an index must exist because $K\setminus X\neq \es$.) We necessarily have $|H_k\setminus X| = 1$.
Defining $X_2 =( \cup \HHH) \cup H_k$, we obtain
$X_1= X \subset X_2\SB X\cup Y$ with $|X_2\setminus X_1| = 1$.
An induction completes the proof.
\end{proof}

Note however that the base of a $\cup$-closed wg-family need not be well-graded.
      
\begin{example}
\label{base +wg base}
The $\cup$-closed wg-family
\begin{gather}\nonumber
\hskip -2cm \FFF=\{\es, \{a\},\{b\},\{c\}, \{a,b\},\{a,c\},\{b,c\},\{c,d\}, \{a,b,c\},\\ \label{counter wg base}
\hskip 2cm \{a,c,d\}, \{b,c,d\},\{a,b,c,d\},
 \{a,b,c,d,e\}\}.
\end{gather}
has the base
$
\{\es, \{a\},\{b\},\{c\},\{c,d\},\{a,b,c,d,e\}\},
$
which is not well-graded. Moreover, $\FFF$ has two different minimal well-graded subfamilies spanning $\FFF$: 
\begin{gather}\nonumber \hskip -3cm
\{\es,  \{a\},  \{b\},  \{c\}, \{a,b\},\{a,c\}, \{c,d\},\{a,b,c\},\\
\label{counter wg base1} \hspace{3cm} \{a,c,d\}, \{a,b,c,d\},\{a,b,c,d,e\}\},\\ \nonumber \hskip -3cm
\{\es, \{a\},\{b\},\{c\}, \{a,b\},\{b,c\}, \{c,d\},\{a,b,c\},\\ \label{counter wg base2}\hspace{3cm} \{b,c,d\}, \{a,b,c,d\},\{a,b,c,d,e\}\}.
\end{gather}
\end{example}

\begin{example}
\label{u-i-closed base not wg}
Notice that the base of a family which is both $\cup$-closed and $\cap$-closed (that is, closed under intersection) is not necessarily well-graded. Indeed, consider the family
\begin{gather*}
\hskip -1cm \GGG= \{\es, \{a\},\{b\},\{d\},\{a,b\}, \{a,d\},\{b,d\}, \{a,b,c\},\{a,b,d\},\\
\hskip 7cm \{a,b,c,d\},\{a,b,c,d,e\}\},
\end{gather*}
for which
$
\{\es, \{a\},\{b\},\{d\},\{a,b,c\},\{a,b,c,d,e\}\}
$
is the base.
\end{example} 

\section*{Main Results}

\begin{theorem}
\label{base => wg}
Let $\FFF$ be a $\cup$-closed family with base $\BBB$. Then $\FFF$ is a wg-family if and only if,
for any two distinct sets $K$ and $L$ in $\BBB$, there is a tight path in $\FFF$ 
from $K$ to $L\cup K$. If $\BBB$ contains the empty set, then $\FFF$ is well-graded if and only if  there is a tight path from $\es$ to $K$ for any $K$ in $\BBB$.
\end{theorem}
Thus, this result provides a solution to Problems 1 and 2. Another solution to Problem 2 is given by Lemma \ref{lem:wg with empty by endpoints}.
\tl
\begin{proof}
As $\FFF$ is $\cup$-closed with base $\BBB$, the necessity is clear for both statements. To establish that the sufficiency in the first statement also holds, we point out that the family $\BBB^*$ defined by
\begin{gather}\label{build BBB*}
M\in \BBB^*\,\,\EQ\,\,\begin{cases}  M = \cup \AAA\text{ for some }\AAA\SB \BBB\text{ such that}\\
 K\SB \cup \AAA\SB K\cup L \text{ for some }K,L\in\BBB
 \end{cases}
\end{gather}
includes $\BBB$  since $K = \cup \{K\}$ and $K\SB\cup\{K\}\SB K\cup L$ for any $K$ and $L$ in~$\BBB$. 
Since $\BBB \SB \BBB^*\SB \FFF$ the family $\BBB^*$ spans $\FFF$. 
We claim that $\BBB^*$ is well-graded, which implies by Lemma \ref{span is wg} that $\FFF$ is well-graded.   The main line of our argument is similar to that used in the proof of Lemma \ref{span is wg}.

Take any two distinct $V,W\in\BBB^*$. By definition of $\BBB^*$, we have $V = \cup \VVV$ and $W =\cup\WWW$ for some subfamilies $\VVV$ and $\WWW$ of $\BBB$. Suppose that $d(V,V\cup W)=n$. 
We have to show that there exists in $\BBB^*$ a tight path 
$$
V_0 = V, V_1,\ldots,V_n = V\cup W
$$
from $V$ to $V\cup W$.
Without loss of generality (exchanging the roles of $V$ and $W$ if needed), we can assume that there is some $H\in\WWW$ such that $H\setminus V\neq\es$. Choose $G\in \VVV$ arbitrarily. Then $G\subset H\cup G \SB V\cup W$, with $H$ and $G$ in~$\BBB$. By hypothesis, there is  a tight path $G_0 = G, G_1,\ldots,G_m = G\cup H$ from~$G$ to $G\cup H$ in $\FFF$, with $G \subset G_i\subset G\cup H$ and 
$d(G,G_i) = i$ for $1\leq i \leq m$.
Let $k$ be the first index such that $G_k\setminus V\neq \es$.
(Such an index must exist because $H\setminus V\neq \es$.) We necessarily have $|G_k\setminus V| = 1$.
Defining $V_1 =( \cup \VVV) \cup G_k$, we obtain
$V_0= V \subset V_1\SB V\cup W$ with $|V_1\setminus V_0| = 1$.
An induction completes the proof of the sufficiency for the first statement. 

We now show that if $\es\in\BBB$, then there is a tight path from $L$ to $K\cup L$ for any $K$ and $L$ in $\BBB$. Thus, the sufficiency of the second statement follows from that in the first statement. Indeed, let $K_0=\es,K_1,\ldots,K_n =K$ be a tight path. It is easily seen that, after removal of identical terms if need be, the sequence $K_0\cup L = L$, $K_1\cup L, \ldots,K_n\cup L= K\cup L$ is a tight path from $L$ to $K\cup L$.
\end{proof}

\begin{remark}
The set $\BBB^*$ constructed in the proof of Theorem 
\ref{base => wg} is not necessarily a minimal wg-family spanning 
$\FFF$. Indeed, the definition of $\BBB^*$ by (\ref{build BBB*}) includes all the unions $\cup\AAA$, while only some of them may be needed. An example  was provided by the wg-family of Example \ref{base +wg base}. In this case, 
each of (\ref{counter wg base1}) and (\ref{counter wg base2}) is a minimal wg-family including the base and spanning the wg-family $\FFF$ defined by (\ref{counter wg base}). 
The set $\BBB^*$ in this case would be the union of the two families in
(\ref{counter wg base1}) and (\ref{counter wg base2}), which is  in fact equal  to $\FFF$.  
\end{remark}

In the case of learning spaces, \citet{koppen:98} obtained a different, but equivalent answer to Problem 1 (see Theorem \ref{koppen theo}). As shown by Counterexample \ref{counter gen mathieu}, Koppen's result does not generalize to the case in which the family does not contain the empty set. 
We review this result below. To this end, we recall some concepts and results of \citet{doign99}, which we adapt to the general case in which the empty set is not assumed to belong to the family\footnote{All of \citet{doign99}'s results were developed in the context of knowledge spaces, that is $\cup$-closed families containing the empty set. We drop the latter condition here.}.  Even though the proofs of Theorems \ref{atom iff} and \ref{base = atoms} are essentially those of Theorems 1.25 and 1.26 in \citet{doign99}, we include those proofs for completeness because our context is more general.

\begin{definition}
\label{atoms}
For any $x$ in $\XXX = \cup\FFF$, where $\FFF$ is a $\cup$-closed family, an {\sl atom at} $x$ is a minimal set of $\FFF$ containing $x$ (where `minimal' is with respect to set inclusion). A set $X$ in $\FFF$ is called an {\sl atom}\footnote{Our meaning of the term `atom'  is different from its usage in lattice theory; cf.~\citet{birkh67}, \citet{davey90}. It also slightly differs from that in \citet{doign99} because we do not allow the empty union of a family (see Footnote 4, 5 and 6).}  if either $X = \es\in\FFF$, or there is some $x\in\XXX$ such that $X$ is an atom at $x$. Writing $\POW(\FFF)$ for the power set of $\FFF$, we denote by 
$\ss(x)$ the collection of all the atoms at $x$ and refer to $\ss: \XXX \to \POW(\FFF)$ as the {\sl surmise function} of $\FFF$. Clearly, since $\XXX$ is finite, we have $\ss(x)\neq\es$ for every $x\in\XXX$; thus, there is at least one atom at every point of $\XXX$. (But see Counterexample \ref{counter wg no atom}.)
\end{definition}

\begin{theorem}
\label{atom iff}
A nonempty set $X$ in a $\cup$-closed family $\FFF$ is an atom if and only if $X\in\HHH$ for any subfamily $\HHH$ of $\FFF$ satisfying $\cup \HHH = X$.
\end{theorem}

\begin{proof}
(Necessity.) Suppose that $X$ is an atom at some $x \in \cup\FFF$, with $X =\cup \HHH$ for some subfamily $\HHH$ of $\FFF$. Then $x\in Y$ for some $Y\in\HHH$, with necessarily $Y\SB X$. This implies $Y = X$ because $X$ is a minimal set containing $x$, and so $X\in\HHH$. The case of the atom $\es$ is straightforward.

(Sufficiency.) If some $X\in\FFF$ is not an atom, then for each $x\in X$, we must have 
$x\in Y(x) \subset X$ for some $Y(x) \in\FFF$. Writing $\HHH = \{Y(x)\st x\in X\}$, we get $\cup\HHH =  X$, with $X\notin \HHH$.
\end{proof}

\begin{theorem}
\label{base = atoms}
The base of a   $\cup$-closed family $\FFF$ is the collection of all its atoms.
\end{theorem} 

\begin{proof}
Let $\AAA$ be the collection of all the atoms of $\FFF$.   We claim that $\AAA$ must be the base of $\FFF$.  If $\es\in\FFF$,  we have $\es\in\AAA$ by definition with $\cup\{\es\} = \es$.
Notice that, for any $X\neq\es$ in $\FFF$, the set  $\AAA_X=\{Y\in\AAA\st\exists x\in X,\,\, x\in Y\SB X\}$ exists because there is an atom at every point of $X\SB \XXX$. We have thus $\cup \AAA_X = X$ and so $\AAA$ spans $\FFF$ (whether or not $\es\in\FFF$).  Let now $\HHH$ be another subfamily of $\FFF$ spanning $\FFF$. Take any $Z\in \AAA$. Since $\HHH$ spans $\FFF$, there must be a subfamily $\GGG$ of $\HHH$ such that $\cup\GGG = Z$. By Theorem \ref{atom iff}, we must have $Z\in\GGG\SB\HHH$; this yields $\AAA\SB\HHH$. Thus, $\AAA$ is a minimal family spanning $\FFF$ and so is the (unique) base of $\FFF$. \end{proof}

Note in passing that, in the infinite case, there may not be an atom at every point of the ground set $\XXX = \cup\FFF$ of a $\cup$-closed family $\FFF$. We already gave the example of the collection of all the open sets of $\Ree$. Below is another, simple example.

\begin{counterexample}{\rm
\label{counter wg no atom}
Consider the infinite family $\FFF = \GGG + \HHH$, with
\begin{align}\label{counter GGG}
\GGG &= \{G_n\st G_n = \{\ldots,\frac 1{n+1},\frac 1n\},\,n>1\}
\\ \label{counter HHH}
\HHH&=  \{H_n\st H_n = G_n +\{1\},\,G_n\in\GGG\}.
\end{align}
The family $\FFF$ is $\cup$-closed and well-graded and  there is no atom at $1$. It is easily verified that this $\cup$-closed family $\FFF$ has no base. The $\cup$-closed family $\HHH$ is its own base and has no atom at 1 either.}
\end{counterexample}

We turn to \citet{koppen:98}'s result, which is formulated as the last statement in the theorem below. We recall that $\check \BBB = \BBB\setminus\{\es\}$ for the base $\BBB$ of a learning space.

\begin{theorem}
\label{koppen theo}
Suppose that $\FFF$ is a learning space with base $\BBB$ and surmise function~$\ss$. Then $\{\ss(x)\st x\in\cup \FFF\}$ is a partition of $\check\BBB$ if and only if there is a tight path from $\es$ to $K$ for any $K\in\check\BBB$. Accordingly, $\FFF$ is well-graded if and only if $\{\ss(x)\st x\in\cup \FFF\}$ is a partition of $\check\BBB$.
\end{theorem}

\begin{proof}  
{\bf Observation.} By Theorem \ref{base = atoms}, for any $K\in \check\BBB$, there is some $y\in K$ such that $K$ is an atom at $y$. Moreover, the hypothesis that $\{\ss(x)\st x\in\cup \FFF\}$ is a partition of $\check\BBB$ and $|K|>1$ implies that there exists, for any $y'\in K$ distinct from $y$, at least one atom at $y'$ strictly included in~$K$. (Otherwise, we would have $K\in \ss(y)\cap\ss(y')$.) 
\vtl 
Assume that $\{\ss(x)\st x\in\cup \FFF\}$ is a partition of $\check\BBB$. Take any $K$ in $\check\BBB$ and suppose that  $|K| =n$. By the Observation, $K_n=K$ is an atom at $x_n$ for some $x_n\in K$.  We use induction on~$n$.
If $n=1$, then $\es,\{x_1\}=K$ is the tight path. Suppose that we have a tight path $K_0=\es, K_1 = \{x_1\},\ldots,K_j = \{x_1,\ldots,x_j\}$ from $\es$ to $K_j\subset K_n = K$.  From  the
Observation, we know that there exists an atom $L_\ell$ at $y_\ell$ for any $y_\ell\in K\setminus K_j$, with $1\leq \ell\leq n-j$ and $L_\ell\subset K$.   If $|K_j\cup L_\ell |> j+1$ for some index $\ell$, then, again by the Observation, there is some index $i\neq \ell$, $1\leq i\leq n-j$,  
such that $L_i\subset L_\ell$ is an atom at $y_i\in L_\ell\setminus K_j$, with 
$j+1\leq |K_j \cup L_i|< |K_j \cup L_\ell |$. By elimination, we have necessarily some $y_k\notin K_j$, $1\leq k\leq n-j$
and an atom $L_k \subset K$ at $y_k$ such that $K_j \cup \{y_k\} = K_j\cup L_k$. Defining $x_{j+1} = y_k$ and $K_{j+1} = K_j \cup \{x_{j+1}\}$, we obtain the tight path $K_0= \es,K_1,\ldots,K_{j+1}$ from $\es$ to $K_{j+1}\SB K$. Applying induction yields the necessity in the first statement. 

Conversely, assume that there is a tight path from $\es$ to $L$ for any $L\in \BBB$. Suppose that $\es\neq K\in \ss(x)\cap \ss(y)$ for some $K\in\check\BBB$ and some distinct $x,y\in\cup\FFF$. A contradiction ensues because no tight path
$$
K_0=\es,K_1= \{x_1\},\ldots, K_n =\{x_1,\dots,x_n\}=K
$$ 
from $\es$ to $K$ can exist. Indeed, we must have $x,y\in K\setminus K_{n-1}$ since $K$ is an atom at both $x$ and $y$; and yet $|K| = n > n-1 = |K_{n-1}|$.

The last statement of the theorem follows from the last statement in Theorem \ref{base => wg}.\end{proof}

As announced, this result does not generalize to the case in which the family $\FFF$ does not contain the empty set, even if we assume that the family is {\sl discriminative} that is, satisfies the condition: for all $x,y\in \cup\FFF$ 
$$
(\forall X\in\FFF, \,\, x\in X\,\eq\,y\in X)\,\,\EQ\,\,x = y.
$$
\begin{counterexample} \label{counter gen mathieu}{\rm Consider the family $\KKK$ defined by the base
\begin{equation*}
\AAA= \{\{x,y,c\},\{y,d\},\{c,d\}\}.
\end{equation*}
We get the surmise function
\begin{gather*}
\ss(x)= \{\{x,y,c\}\}, \quad
\ss(y) = \{\{x,y,c\},\{y,d\}\},\\
\ss(c)=\{\{x,y,c\},\{c,d\}\},\quad
\ss(d)=\{\{y,d\},\{c,d\}\}.
\end{gather*}
}
\end{counterexample}
It is easily checked that $\KKK$ is discriminative and well-graded; yet, the surmise function does not define a partition of the base $\AAA$.

\section*{Algorithms}

In the algorithms described in this section, we are given as input a family of sets $\BBB$, which is purported to be the base of a $\cup$-closed family $\FFF$. We wish to test whether this is true, and if so to determine other properties of $\FFF$ such as whether it is well-graded or a learning space. In many cases the definitions given in earlier sections of this paper may already be directly translated into algorithms, but a definition may be translated into an algorithm in multiple ways, some more efficient than others; the content of the results lies less in the pure existence of the algorithms and more in designing the algorithms so they perform their tasks efficiently and in analyzing how much time they take to run. The time for our algorithms should be polynomial in the size of our input, if possible; this size is the sum of the cardinalities of the sets in $\BBB$. In particular, this requirement for polynomial time precludes explicit construction of $\FFF$ as the span of $\BBB$, as $\FFF$ may have exponentially greater size.

We assume a standard random-access-machine model of computation in which simple arithmetic steps and memory access operations may be performed in constant time. The input to our algorithms will be families of sets. We assume that each set element is represented as an object that takes a constant amount of computer storage and with which additional information may be associated. For instance, a natural representation with these properties would be to represent the $n$ elements of a set family as integers in the range from $0$ to $n-1$; we may then associate information with each element by using these integers as array indices. We represent an input set as a list of elements, and an input set family as a list of lists of elements. As is standard in the analysis of algorithms, we use $O$-notation to simplify the stated time bounds for our algorithms.

\begin{definition}
In order to analyze and compare the running times of our algorithms, we need parameters to describe the input size. We define $n$ to be the number of sets in $\BBB$, $\ell$ to be the size of the largest set in $\BBB$, and $m$ to be the sum of cardinalities of sets in $\BBB$. We say that an algorithm runs in polynomial time if its worst-case running time can be upper bounded by a polynomial function of $\ell$, $m$ and $n$. For purposes of comparing run times it is convenient to note that $\ell\le m\le n\ell$.
\end{definition}

\begin{definition}
The \emph{endpoints} of a set $X$ belonging to a base $\BBB$ are the elements of the set 
$$
X\setminus\bigcup_{Y\in \BBB, Y\subset X} Y.
$$
That is, the endpoints of $X$ are the elements of $X$ that are not contained in any proper subset of $X$ that belongs to $\BBB$. Equivalently, $x$ is an endpoint of a set $X$ in a base $\BBB$ if $X$ is an atom at $x$.
\end{definition}

\begin{lemma}
\label{lem:endpoint alg}
There is an algorithm that takes as input a set family $\BBB$ and a set $X\in \BBB$, and that outputs the endpoints of $X$, using time $O(m)$.
\end{lemma}

\begin{proof}
We associate with each element $x$ in $\cup \BBB$ a Boolean variable that is true if and only if $x$ is in $X$; setting up these variables takes time $O(m)$. By examining the value for $x$, we may test whether $x$ belongs to $X$ in constant time. For each set $Y\in\BBB$, we use these bits to determine whether $Y\subset X$, by testing each of the members of $Y$, in time $O(|\cup Y|)$. By performing this test for all sets in $\BBB$, we may determine a collection of the subsets of $X$ that are in $\BBB$, in total time $O(m)$. We then associate a second Boolean variable with each member of $X$; initially we set all of these variables to false. For each $Y$ in our collection of subsets of $X$, we loop through the elements of $Y$, and set the Boolean variables associated with each of these elements to true. Finally, we loop through the elements of $X$, and form a list of the elements for which the associated Boolean value remains false. These elements are the endpoints of $X$. The runtime of this algorithm is dominated by the steps in which we find the subsets of $X$ and then use those subsets to mark covered elements of $X$; both of these steps take $O(m)$ total time.
\end{proof}

\begin{theorem}
\label{thm:test if base}
Given a family $\BBB$ of sets, we may determine in time $O(nm)$ whether $\BBB$ is the base of a $\cup$-closed family $\FFF$.
\end{theorem}

\begin{proof}
We use Lemma~\ref{lem:endpoint alg} to calculate the endpoints of each $X\in\BBB$. By definition, $X$ is an atom if and only if it is empty or has a nonempty set of endpoints; thus, by Theorem~\ref{base = atoms}, $\BBB$ is the base of its span if and only if every set in $\BBB$ is either empty or has a nonempty set of endpoints. There are $n$ sets, each of which takes time $O(m)$ to test, so the total time is $O(nm)$.
\end{proof}

\begin{lemma}
\label{lem:wg with empty by endpoints}
Suppose that set family $\BBB$ contains the empty set. Then $\BBB$ is the base of a $\cup$-closed well-graded family if and only if each nonempty $X\in\BBB$ has one endpoint.
\end{lemma}

This is closely related to some results of \citet{koppen:98} \citep[see also][Theorem 3.15, Condition (ii)]{doign99}.
\tl

\begin{proof}
If some $X\in\BBB$ has two or more endpoints $x$ and $y$, then $X$ belongs to both $\sigma(x)$ and $\sigma(y)$, so the surmise function $\sigma$ is not a partition of $\check\BBB$. If some nonempty $X$ has no endpoint, it is not an atom and not part of a base. Conversely if every nonempty $X\in\BBB$ has one endpoint, then $\sigma$ partitions $\check\BBB$ according to those endpoints. The result follows from Theorem~\ref{koppen theo}.
\end{proof}

\begin{theorem}
\label{thm:test ls}
Given a family $\BBB$ of sets, we may determine in time $O(nm)$ whether $\BBB$ is the base of a learning space.
\end{theorem}

\begin{proof}
We first check that $\BBB$ contains the empty set; if not, it cannot be the base of a learning space.
Then, as in Theorem~\ref{thm:test if base}, we apply Lemma~\ref{lem:endpoint alg} to calculate the endpoints of each $X\in\BBB$. By Lemma~\ref{lem:wg with empty by endpoints}, $\BBB$ is the base of a $\cup$-closed well-graded family if and only if each nonempty $X\in\BBB$ has exactly one endpoint. There are $n$ sets, each of which takes time $O(m)$ to test, so the total time is $O(nm)$.
\end{proof}

\begin{definition}
For any set family $\BBB$ and any set $X\in\BBB$, let $\BBB/X$ denote the family of sets $\{Y\setminus X\st Y\in\BBB\}$.
\end{definition}

\begin{lemma}
\label{lem:wg to ls reduction}
Let $\BBB$ be the base of a $\cup$-closed family $\FFF$. Then $\FFF$ is well-graded if and only if, for each $X$ in $\BBB$, the family $\BBB/X$ spans a learning space.
\end{lemma}

\begin{proof}
A tight path in $\FFF$ from $X$ to some set $Y\supset X$ corresponds (via set-theoretic difference of each path member with $X$) to a tight path in $\FFF/X$ from the empty set to $Y\setminus X$. Conversely, a tight path in $\FFF/X$ from the empty set to $Y\setminus X$ corresponds (via set-theoretic union of each path member with $X$) to a tight path in $\FFF$ from $X$ to $Y$. The result follows from Theorem~\ref{base => wg}.
\end{proof}

\begin{theorem}
\label{thm:test wg base}
Given a family $\BBB$ of sets, we may determine in time $O(n^2m)$ whether $\BBB$ is the base of a $\cup$-closed well-graded family.
\end{theorem}

\begin{proof}
We may first test whether $\BBB$ is a base by Theorem~\ref{thm:test if base}.
Next, for each $X\in\BBB$, we form the set $\BBB_X$ consisting of the empty set and the sets in $\BBB/X$ that have a nonempty set of endpoints with respect to $\BBB/X$, and test whether $\BBB_X$ is the base of a learning space by Theorem~\ref{thm:test ls}. $\BBB$ itself is the base of a $\cup$-closed well-graded family if and only if each $\BBB_X$ passes this test, by Lemma~\ref{lem:wg to ls reduction}. There are $n$ sets $\BBB_X$, each takes time $O(nm)$ to construct and test, and so the total time bound is $O(n^2m)$.
\end{proof}

We now consider the situation in which $\BBB$ is not itself the base of a well graded family. Can we modify $\BBB$ to produce a well graded family that is as close as possible, in some sense, to the span of $\BBB$?

\begin{definition}
A \emph{minimal well-graded extension} of a family of sets $\BBB$ is a well-graded $\cup$-closed set family $\FFF$ such that $\BBB\subset\FFF$, and such that no $\cup$-closed $\FFF'$ with $\BBB\subset\FFF'\subset\FFF$ is well-graded. A \emph{path family} for a family of sets $\BBB$ is a set  paths $\pi_{K,L}$ with $K$ and $L$ in $\BBB$; $\pi_{K,L}$ may use sets not belonging to the span of $\BBB$, but is required to be a tight path in the power set of $\cup\BBB$. We observe that the length of a path $\pi_{K,L}$ is at most the cardinality of $L$, and therefore that the total length of all paths in a path family is $O(nm)$. A \emph{path extension} $\FFF$ of $\BBB$ is formed from a path family by letting $\BBB'$ consist of all the sets occurring on paths $\pi_{K,L}$ and letting $\FFF$ be the span of $\BBB'$.
\end{definition}

\begin{lemma}
\label{lem:path extension wg}
Any path extension is well-graded.
\end{lemma}

\begin{proof}
We show that, for every $K$ and $L$ in $\BBB'$, where $\BBB'$ is the family of sets occurring on paths $\pi_{X,Y}$ in a path extension of $\BBB$, that there exists a tight path in the span of $\BBB'$ from $K$ to $K\cup L$.

Thus, suppose $K$ belongs to a path $\pi_{A,B}$ and $L$ belongs to a path $\pi_{C,D}$. To form a tight path from $K$ to $K\cup L$ in $\FFF$, we concatenate the following three paths:
\begin{enumerate}
\item a tight path from $A$ to $K$ along path $\pi_{A,B}$,
\item a tight path from $K$ to $K\cup C$, formed by the union of $K$ with the sets in path $\pi_{A,C}$, and
\item a tight path from $K\cup C$ to $K\cup L$, formed by the union of $K\cup C$ with the sets in the portion of path $\pi_{C,D}$ that extends from $C$ to $L$.
\end{enumerate}
When this concatenation would cause the same set to appear repeatedly, we discard the duplicate sets.
It is straightforward to verify that each set in this concatenation of paths belongs to the span of $\BBB'$. Thus, we can form a tight path from any $K$ to $K\cup L$ in this span, and therefore, by Theorem~\ref{base => wg}, the span is well-graded.
\end{proof}

\begin{lemma}
\label{lem:minimal-is-path}
Any minimal well-graded extension is a path extension.
\end{lemma}

\begin{proof}
Let $\FFF$ be a minimal well-graded extension of $\BBB$. Then by Theorem~\ref{base => wg} we can find a path family for $\BBB$, such that each set occurring in each path belongs to $\FFF$.  By Lemma~\ref{lem:path extension wg}, the corresponding path extension is well-graded, and it is a subfamily of $\FFF$ and contains every set in $\BBB$. By the minimality of $\FFF$, this path extension must coincide with $\FFF$.
\end{proof}

It is straightforward to combine the results above in an algorithm that finds a minimal well-graded extension of any set family in polynomial time: construct a path family arbitrarily, and then for each set in its base, determine whether the set can be removed by testing the result of the removal for well-gradedness, using Theorem~\ref{thm:test wg base} to do these tests. However the polynomial time bound of this algorithm would be large. We now describe a more efficient algorithm for the same task, based on a more careful choice of path family.

\begin{theorem}
\label{thm:minimal-alg}
Given any family of sets $\BBB$, we can find a minimal well-graded extension of $\BBB$ in time $O(nm\ell+n^3m)$.
\end{theorem}

\begin{proof}
We simultaneously form the paths $\pi_{K,L}$ in a path family, and a superset $\BBB'$ of $\BBB$ that includes a base for our eventual well-graded extension, by adding sets in order by the cardinality of the sets. At step $i$ of the process, we include sets of cardinality $i$ into $\BBB'$, taking care as we do that all the sets we add are necessary for well-gradedness.  As we do so, we maintain the following data:
\begin{itemize}
\item $S_{K,L}$ is the union of all sets $X\in\BBB'$ such that $X\subset K\cup L$ and $X\cup K\ne K\cup L$.
\item $c_{A,B,C,D}$ is a Boolean value, true if and only if $S_{A,B}\subset C\cup D$.
\end{itemize}

In the $i$th step of the algorithm, we consider the set $\Pi_i$ of all paths $\pi_{K,L}$ such that $|S_{K,L}|=i-1$ and such that $|K\cup L|>i$. We will add sets of cardinality $i$ to $\BBB'$, of the form $S_{K,L}\cup\{x\}$ for some $x$ in $(K\cup L)\setminus S_{K,L}$, in order to allow one more step on each path. However, note that if $S_{A,B}=S_{C,D}$ then a single set of this type may allow an additional step for multiple paths.

We observe that, if $\pi_{A,B}$ and $\pi_{C,D}$ are both in $\Pi_i$, then $c_{A,B,C,D}$ is true if and only if $S_{A,B}=S_{C,D}$. Thus the relation $c$ can be viewed as an equivalence relation on the paths in $\Pi_i$. As part of our calculation in step $i$ of the algorithm, we construct the equivalence classes of this equivalence relation.

For each equivalence class, we form a bipartite graph $(U,V,E)$. Here $U$ consists of pairs $(K,L)$ corresponding to paths $\pi_{K,L}$ in the equivalence class.  $V$ consists of elements $x$ in the sets $(K\cup L)\setminus S_{K,L}$, for paths $\pi_{K,L}$ in the equivalence class. We draw an edge from $(K,L)$ to $x$ if $x\in (K\cup L)\setminus S_{K,L}$. We find a minimal subset of $V$ that dominates every vertex of $U$ in this graph; this gives us a minimal family of sets that we can add to $\BBB'$ in order to take another step on each path in the equivalence class. This minimal dominating set can be found by repeatedly either including in it a vertex in $V$ that is the only neighbor of some vertex in $U$ or, if no such vertex in $U$ exists, removing from the graph an arbitrarily chosen vertex in $V$; the total time to perform this step is proportional to the size of the graph.

Once we have found these sets to add to $\BBB'$, we must update the data we are maintaining so that we may repeat this computation for a larger value of $i$. Whenever we add a set corresponding to a path $\pi_{A,B}$ in $\Pi_i$ and element $x$, we examine all paths $\pi_{C,D}$ for which $c_{A,B,C,D}$ is true. If $x\in C\cup D$, we include $x$ as a new member of $S_{C,D}$. (In particular, $c_{A,B,A,B}$ will always be true, and we will always include $x$ as a new member of $S_{A,B}$.) However, if $x\notin C\cup D$, we instead set $c_{A,B,C,D}$ to false.

We now analyze the running time of this algorithm:
\begin{itemize}
\item We may compute the initial value of each set $S_{K,L}$ in time $O(m)$, simply by testing each other set $X$ in time $O(|X|)$, after an initial $O(|K|+|L|)$ time preprocessing stage to construct data structures for testing membership in $K\cup L$. and $L\setminus K$. Thus, we may construct all such sets in time $O(n^2m)$.
\item We may compute the initial value of $c_{A,B,C,D}$ in time $O(|A|+|B|+|C|+|D|)$. Adding this up over all quadruples $A,B,C,D$ produces a runtime of $O(n^3m)$.
\item Identifying $\Pi_i$ takes time $O(n^2)$. There are $O(m)$ steps of the algorithm, so the total time for this identification is $O(n^2m)$.
\item The sum of the cardinalities of the sets $\Pi_i$, summed over all $i$, is $O(nm)$, because each time we include a set in $\Pi_i$ we take a step on the corresponding path, and the total length of all paths in a path family is $O(nm)$. We may identify the equivalence class of a single path in time $O(n^2)$; therefore, the total time to construct equivalence classes, throughout the course of the algorithm, is $O(n^3m)$.
\item Each vertex in $U$ in the bipartite graph constructed for an equivalence class may have $O(\ell)$ neighbors. Therefore, the total size of all bipartite graphs so constructed, and the total time to find dominating sets in these graphs, is $O(nm\ell)$.
\item Each set added to $\BBB'$ can be constructed explicitly, as a list of elements, from the data structures we already have, in time $O(\ell)$. Thus, the total time to list all these sets is $O(nm\ell)$. In addition, for each such set, we spend $O(n^2)$ time examining the paths for which $c_{A,B,C,D}$ is true, the total for which over the course of the algorithm is $O(n^3m)$.
\end{itemize}
Thus, the total time for all of these steps is $O(nm\ell+n^3m)$.
\end{proof}

It may be seen as a flaw in the completion algorithm described above that not every set in the input family $\BBB$ is necessarily part of a base of the output family $\FFF$ it produces.  Given $\BBB$, can we find a wg-family $\FFF$ such 
that every set in $\BBB$ is part of the base of $\FFF$? Unfortunately, as we now show, this 
problem appears to be intractable.

\begin{theorem}
It is NP-complete, given a set family $\BBB$, to determine whether there exists a well-graded $\cup$-closed set family $\FFF$ such that $\BBB$ is a subset of the base of $\FFF$.
\end{theorem}

\begin{proof}
If an $\FFF$ satisfying this requirement exists, we can choose $\FFF$ to be minimal and therefore, by Lemma~\ref{lem:minimal-is-path}, a path extension. Thus, we can test in NP whether $\FFF$ exists by nondeterministically choosing a path extension, applying Lemma~\ref{lem:endpoint alg} to find the endpoints of all sets included on paths in the extension, and verifying that each member of $\BBB$ has an endpoint. Therefore, determining whether $\FFF$ exists belongs to NP, the easier part of proving that it is NP-complete.

To finish the NP-completeness proof, we reduce the problem from a known NP-complete problem, 3-satisfiability \citep{GareyJohnson}. A 3-satisfiability instance consists of sets of \emph{variables} $V$, complements of variables $\bar V=\{\bar v\st v\in V\}$, and a set $C$ of clauses, where each clause is a set of three terms, and where a term is any element of $V\cup\bar V$.
A truth assignment is any function $f$ from $V$ to $\{0,1\}$; we may extend $f$ to the domain $V\cup\bar V$ by $f(\bar v)=1-f(v)$. A truth assignment is \emph{satisfying} if each clause of $C$ contains at least one term mapped by $f$ to $1$, and a 3-satisfiability instance is \emph{satisfiable} if and only if it has a satisfying assignment.

From a 3-satisfiability instance $(V,\bar V,C)$ we form a set family $\BBB$, the ground set of which will be the terms and clauses of the instance: $\cup\BBB=V\cup\bar V\cup C$. For each variable $v\in V$, we include in $\BBB$ the set $\{v,\bar v\}$, and for each clause $c$ corresponding to the conjunction of three terms $u$, $v$, and $w$ we include in $\BBB$ two sets, $\{c\}$ and $\{c,u,v,w\}$. Additionally, we include in $\BBB$ the empty set. As we now show, the resulting set $\BBB$ forms a subset of the base of a well-graded $\cup$-closed set family $\FFF$, if and only if the given 3-satisfiability instance is satisfiable.

In one direction, suppose we have a satisfying truth assignment $f$.
Let $\HHH=\{\{x\}\st x\in V\cup\bar V\mbox{ and }f(x)=0\}$. Let $t_i(c)$, for $i\in\{0,1,2\}$ map clauses to terms in such a way that $c=\{t_0(c),t_1(c),t_2(c)\}$ and $f(t_2(c))=1$. Let $T_0=\{\{c,t_0(c)\}\st c\in C\}$ and $T_1=\{\{c,t_0(c),t_1(c)\}\st c\in C\}$.
We let $\FFF$ be the span of $\HHH\cup\BBB\cup T_0\cup T_1$.
For each set $\{v,\bar v\}\in\BBB$ in there is a tight path in $\FFF$ from the empty set through $\{v\}$ or $\{\bar v\}$ respectively as $v$ or $\bar v$ is mapped by $f$ to $0$; the element of $\{v,\bar v\}$ not mapped to $0$ is an endpoint of $\{v,\bar v\}$. For each set $\{c,t_0(c),t_1(c),t_2(c)\}\in\BBB$ there is a tight path in $\FFF$ from the empty set through sets $\{c\},\{c,t_0(c)\},\{c,t_0(c),t_1(c)\}$, and $t_2(c)$ is an endpoint. The paths through $\{c\},\{c,t_0(c)\},\{c,t_0(c),t_1(c)\}$ also form tight paths in $\FFF$ to each set in $T_0$ and $T_1$. Thus, every set in $\HHH\cup\BBB\cup T_0\cup T_1$ has a tight path in $\FFF$ from the empty set, so by Theorem~\ref{base => wg} $\FFF$ is well-graded, and every nonempty set in $\BBB$ has an endpoint, so by Lemma~\ref{lem:wg with empty by endpoints} $\BBB$ is part of the base of $\FFF$. Thus, we have shown that, if $f$ is a satisfying truth assignment, $\BBB$ forms a subset of the base of a well-graded $\cup$-closed family.

In the other direction, suppose that there exists a set family $\FFF$ that is a minimal path extension of $\BBB$ for which $\BBB$ is a subset of the base. Then, in order to have a tight path from the empty set to $\{v,\bar v\}$, while not eliminating that set from the base, $\FFF$ must contain exactly one of the two sets $\{v\}$, $\{\bar v\}$; form a truth assignment $f$ in which we assign $v$ the value 1 if $\{\bar v\}$ is in $\FFF$ and the value 0 if $\{v\}$ is in $\FFF$. This must be a satisfying assignment, for if a clause $c$ had no variable satisfying it then the set $\{c,u,v,w\}$ would have no endpoints and therefore couldn't be part of the base of $\FFF$. Thus, we have shown that, if $\BBB$ forms a subset of the base of a well-graded $\cup$-closed family, then the 3-satisfiability instance $(V,\bar V,C)$ has a satisfying assignment.

We have described a polynomial time many-one reduction from the known NP-complete problem of 3-satisfiability to the problem of testing whether a set family is a subset of a base of a well-graded $\cup$-closed family, and we have shown that the latter problem is in NP. Therefore, it is NP-complete.
\end{proof}

\begin{remark}
Since the family $\BBB$ formed by this reduction contains the empty set,
the same reduction shows that it is also NP-complete to determine whether a given set $\BBB$ is a subset of the base of a learning space.
\end{remark}

\begin{remark}
It is natural to desire, not just a minimal well-graded extension, but an extension that is minimum, either in the sense of having the smallest cardinality as a well-graded set family, in the sense of having the smallest cardinality base, in the sense of having the smallest number of additional sets added to the input family $\BBB$, or in the sense of minimizing the sum of cardinalities of additional sets or of the base. We expect that these problems are computationally intractable, but do not have a hardness proof for them. The problems of minimizing base size at least belong to NP, but minimizing the cardinality of $\FFF$ may not since it involves counting the members of a family of sets that may be exponentially larger than the input.
\end{remark}

\bibliography{grand}
\bibliographystyle{plainnat}
\end{document}